\documentclass[12pt,a4paper]{amsart}

\usepackage{latexsym,amssymb,amsmath}
\usepackage{color}
\usepackage{graphics,xy,epsfig,picture,epic}

\def\cc{{\mathcal C}}

\def\mm{{\mathcal M}}
\def\nn{{\mathcal N}}

\def\ffi{\varphi}
\def\eps{\varepsilon}
\def\dst{\displaystyle}

\renewcommand{\Im}{\mathrm{Im}\,}

\DeclareMathOperator{\vect}{span}
\def\C{{\mathbb{C}}}

\def\N{{\mathbb{N}}}

\def\Q{{\mathbb{Q}}}
\def\R{{\mathbb{R}}}
\def\S{{\mathbb{S}}}

\def\Z{{\mathbb{Z}}}

\def\d{\,{\mathrm{d}}}
\newcommand{\norm}[1]{{\left\|{#1}\right\|}}

\newcommand{\abs}[1]{{\left|{#1}\right|}}
\newcommand{\scal}[1]{{\left\langle{#1}\right\rangle}}

\newtheorem{lemma}{Lemma}[section]
\newtheorem{proposition}[lemma]{Proposition}
\newtheorem{theorem}[lemma]{Theorem}
\newtheorem{corollary}[lemma]{Corollary}

\theoremstyle{definition}
\newtheorem{definition}[lemma]{Definition}

\theoremstyle{remark}
\newtheorem{remark}[lemma]{Remark}

\begin{document}

\title[HUP and the Helmholtz-Laplace equation]
{From Heisenberg uniqueness pairs to properties of the Helmholtz and Laplace equations}

\author{Aingeru Fern\'{a}ndez-Bertolin}
\address{Univ. Bordeaux, IMB, UMR 5251, F-33400 Talence, France.
CNRS, IMB, UMR 5251, F-33400 Talence, France.}
\email{aingeru.fernandez-bertolin@u-bordeaux.fr}

\author{Karlheinz Gr\"ochenig}
\address{Faculty of Mathematics \\
University of Vienna \\
Oskar-Morgenstern-Platz 1 \\
A-1090 Vienna, Austria}
\email{karlheinz.groechenig@univie.ac.at}

\author{Philippe Jaming}
\address{Univ. Bordeaux, IMB, UMR 5251, F-33400 Talence, France.
CNRS, IMB, UMR 5251, F-33400 Talence, France.}
\email{Philippe.Jaming@math.u-bordeaux.fr}
\date{}

\keywords{Unique continuation, Heisenberg Uniqueness Pair, Helmholtz
  equation, Laplace equation, Schwarz
  reflection principle, harmonic functions, nodal set}
\subjclass[2010]{35B05, 35B60, 60H15}

\begin{abstract}
The aim of this paper is to establish uniqueness properties of solutions of the Helmholtz and Laplace equations.
In particular, we show that if two solutions of such equations on a
domain of $\R^d$ agree on two  intersecting $d-1$-dimensional
submanifolds in generic position, then they agree everywhere.
\end{abstract}

\maketitle

\section{Introduction}

The aim of this paper is to bridge two topics: unique determination of 
measures from restrictions of their Fourier transform and
uniqueness properties of solutions of Helmholtz and Laplace 
equations.


Our starting point is the recent notion of Heisenberg uniqueness pairs
introduced by H. Hedenmalm and A. Montes-Rodr\'iguez
\cite{HMR} that we slightly extended in \cite{GJ}:

\begin{definition}
Let $\mm\subset\R^d$ be a manifold and $\Sigma\subset\R^d$ be a set. We say that $(\mm,\Sigma)$
is a {\em Heisenberg uniqueness pair} if the only finite measure $\mu$ supported on $\mm$
with Fourier transform vanishing on $\Sigma$ is $\mu=0$.
\end{definition}

The main focus of \cite{HMR} was 2-dimensional with  $\mm$ being  a hyperbola and $\Sigma$ a discrete set.
The setting was slightly more restrictive as the measure was supposed to be absolutely continuous with respect to
arc length. The property of being  a Heisenberg uniqueness pair has then been established for various curves:
P. Sj\"olin \cite{Sj,Sj2} considered the parabola and the circle, N. Lev \cite{Le} the circle.
Ph. Jaming and K. Kellay \cite{JK} introduced new geometric methods that allowed to treat many curves
when $\Sigma$ consists of 2 intersecting lines.  Those methods were
used in \cite{GS} 
to provide more examples. For results in higher dimensions  we refer to our previous work \cite{GJ}.
For links with the problem of determining point distributions in $\R^d$ (or more generally the determination of finite measures on $\R^d$)
from their projections onto lower dimensional spaces, we refer to \cite{FBGJ}.

So far all  results on the subject treated Heisenberg uniqueness pairs
as a topic in  Fourier analysis related to the uncertainty principle
and with methodological input from  dynamical systems.
In this work we change our point of view and study Heisenberg
uniqueness pairs 
from the perspective of  partial differential 
equations.  To explain the connection to PDEs, we  recall
the results of N. Lev and P. Sj\"olin \cite{Le,Sj} ({\it see} also \cite{JK}).
To  $f \in L^1(\S^1)$ (where $\S^{d-1}$ is the unit sphere of $\R^d$)  we associate
the measure on $\R^2$ given by 
$$
\int \ffi\d\mu=\int_0^{2\pi} \ffi(\cos s,\sin s)f(\cos s,\sin s)\d s,
$$
$\ffi\in\cc(\R^2)$,  i.e., $\mu$ is a measure supported on $\S^1$ and
is absolutely continuous with respect to arc length 
on $\S^1$. 
The Fourier transform of $\mu$ is then defined as
$$
\widehat{\mu}(\xi,\eta)= \int
\exp\bigl(-i(x\eta+y\xi)\bigr)\d\mu (x,y).
$$ 
 Let
$\theta_1,\theta_2\in\R^2$ be two unit vectors and define
$\theta=\arccos\scal{\theta_1,\theta_2}$. Assume that
$\widehat{\mu}=0$ on the lines $\theta_1^\perp$, $\theta_2^\perp$. The
main result of \cite{Le,Sj} asserts that  $\mu=0$ if and only if
$\theta \notin\pi\Q$. 
The connection to partial differential equations is established by the
following observation. If $\mathrm{supp}\, \mu
\subseteq \S ^{d-1}$, then  $u=\widehat{\mu}$ is a solution of the Helmholtz
equation $\Delta u+u=0$. The result of Lev and Sj\"olin can be recast
as follows: If  a solution  
of the Helmholtz equation  is the Fourier transform of a measure
and  vanishes on two   lines whose angle of  intersection is  an irrational 
multiple of $\pi$, then the solution  is identically $0$. 
It turns out that this statement is a special 
case of an older result about nodal sets by S.Y. Cheng \cite[Thm.~2.5]{Ch}.

\medskip

\noindent{\bf Theorem.} (Cheng) {\sl
\label{th:LSGJ}
Let $h$ be a $\cc^\infty$ function on a domain $\Omega\subset\R^2$.
Let $u$ be a solution of $(\Delta +h)u=0$ and assume that two nodal 
lines of $u$ intersect at a point $x_0\in\Omega$,
 i.e.,  there exists two smooth curves
$\Gamma_1,\Gamma_2$ such that $u=0$ on $\Gamma_1\cup\Gamma_2$
with $\Gamma_1\cap\Gamma_2=\{x_0\}$. Then the set of nodal lines through $x_0$ forms an equiangular system. In particular, the angle
between the tangents of $\Gamma_1,\Gamma_2$ at $x_0$ is a rational multiple of $\pi$.}

\medskip

Cheng's proof is rather involved and is based on a subtle  local representation formula of solutions
of elliptic equations due to L. Bers \cite{Be}. 
One of our goals  is to provide two much simpler proofs of Cheng's 
result in the particular case of the Helmholtz ($k>0$) and the Laplace ($k=0$) 
equation $\Delta u+k^2 u=0$.    These proofs are only marginally more
involved than the ones in \cite{Le} and  \cite{Sj} and work under much
less restrictive assumptions.

The first proof is simple and based on
the Schwarz reflection principle for harmonic and analytic functions
and yields the following extension of the result of Lev and Sj\"olin.

\medskip

\noindent{\bf Theorem A.} {\sl 
Let $d\geq 2$ and $\Omega$ be a domain in $\R^d$ with $0\in\Omega$.
Let $\theta_1,\theta_2\in\S^{d-1}$ be such that $\dst\arccos\scal{\theta_1,\theta_2}\notin\pi\Q$.
Let $k\in\R$ and let $u$ be a solution of the Laplace-Helmholtz equation on $\Omega$
$$
\Delta u+k^2 u=0 \, .
$$
If $u$ satisfies  one of the following boundary conditions
$$
\left\{\begin{matrix}u=0&\mbox{on }\theta_1^\perp\cap\Omega\\
u=0&\mbox{on }\theta_2^\perp\cap\Omega\end{matrix}\right.,\quad \text{
or } \quad 
\left\{\begin{matrix}u=0&\mbox{ on } \theta_1^\perp\cap\Omega\\
\partial_nu=0&\mbox{on }\theta_2^\perp\cap\Omega\end{matrix}\right. \, ,
$$
then $u=0$.}

\medskip

We remark that Theorem~A covers general solutions of the Helmholtz
equation, not only solutions that are   Fourier
transforms of a finite measure. Furthermore, Theorem~A allows for a
mixture of Dirichlet and Neumann conditions on the given hyperplanes.  
Finally, the result is valid for  higher dimensions for which part of Cheng's argument is not valid. 




It is natural to  replace  hyperplanes by more general
manifolds. However, in that case, the reflection principle 
becomes substantially more involved,  as it  requires not only  a
point-to-point reflection, but also an additional 
integral operator (see, e.g., \cite{EK,Sa} and references therein). This is a major limitation
to carry such a method to a more general setting.
To overcome this limitation, we pursue a second idea of proof
based on the fact that a solution of the Helmholtz equation 
$\Delta u+k^2u=0$ ($k\not=0$) in a neighborhood of $0$
has an expansion in spherical harmonics of the form
\begin{equation}
\label{eq:helm1}
u(r\theta) \sim(2\pi)^{1/2}(k r)^{-(d-2)/2}\sum_{m=0}^\infty\sum_{j=1}^{N(m)}
a_{m,j} J_{\nu(m)}(k r) Y_m^j(\theta)
\end{equation}
for $r\geq 0$ and $\theta \in S^{d-1}$. Here $J_\nu$ is the Bessel function, $\nu(m) = m + (d - 2)/2$ and $\{Y^j_m\}_{j=1,\ldots,N(m)}$ is a basis for the
spherical harmonics of degree $m$ in $\R^d$. The solution of the
Laplace equation $\Delta u=0$ ($k=0$) possesses a similar local
expansion with $r^m$ in place of the  Bessel function $(k r)^{-(d-2)/2}J_{\nu (m)}(kr)$. 


In dimension 2, we use this formula to show that two Robin conditions
on curves whose  angle of intersection is  an irrational multiple of $\pi$
are essentially incompatible for Laplace-Helmholtz equations
(see Theorem \ref{th:Robin2} for the precise statement).
In this context, we prove a version of Cheng's theorem
for functions that have a Fourier
expansion in polar coordinates of the form
$$
u(r\theta) \sim \sum_{m=-\infty}^\infty c_m\ffi_{\abs{m}}(r)e^{im\theta} \, ,
$$
where the $\ffi_{k}$'s are a family of functions such that
$\ffi_k(r)\sim r^{\alpha_k}$ for some (strictly) increasing sequence of non-negative reals. We call this result the {\em non-crossing lemma }
as it states that nodal lines of such functions cannot cross at an  arbitrary angle.
As a corollary, we obtain a simpler proof of
Cheng's result for solutions of the equation $\Delta u+h(x)u=0$
for radial $h$ (in particular for  constant $h$, we obtain the
 Helmholtz-Laplace equation). Exploiting explicit formulas for bases
 of spherical harmonics, this proof extends to arbitrary  dimension
 and yields  the following statement. 

\medskip

\noindent{\bf Theorem B.} {\sl 
Let $d\geq 3$ and let $\Omega$ be a domain in $\R^d$.
Let $\mm_1,\mm_2$ be two $d-1$ dimensional submanifolds of $\R^d$ that intersect at $0$ and let
$\theta_1,\theta_2$ be the normal vectors  to $\mm_1, \mm _2$  at $0$. Assume that $\arccos\scal{\theta_1,\theta_2}\notin\pi\Q$.

Let $k\in\R$ and let $u$ be a solution of $\Delta u+k^2u=0$ on $\Omega$ such that $u=0$ on $\mm_1\cup\mm_2$. Then $u=0$
on $\Omega$.
}

\medskip

In other words, under the stated conditions on  $\mm_1,\mm_2$  as above, 
$(\S^{d-1},\mm_1\cup\mm_2)$ is a Heisenberg uniqueness pair.

Theorem B is a variation on the unique continuation of solutions of the Helmholtz equation. It states that a solution of that equation is uniquely determined by its restriction to two generic $(d-1)$-dimensional intersecting submanifolds and can therefore be continued to the full domain.

From the point of view of PDEs, it is perhaps more natural to say that a solution of the Helmholtz equation can be uniquely continued from one side of the manifold $\mm_1$ to the other (by some sort of reflection) and likewise for $\mm_2$. Theorem B says that, in general, these continuations are not compatible and thus a solution vanishing on $\mm_1$ and $\mm_2$ vanishes everywhere.

This picture makes it plausible that the extension/continuation of a solution from its restrictions to lower dimensional manifolds, though unique, is inherently unstable. Indeed, we will prove several statements in this regard (see, e.g., Propositions 3.7 and 4.2).

\smallskip

The  paper is organized as follows. In  Section~2, we give two simple
proofs of  Theorem A and prove the non-crossing Lemma and its
corollaries. Section~3 treats the more technical extension of 
the uniqueness properties under  Robin-type  conditions 
in dimension 2.  Section~4 is devoted to some  results in higher dimensions.

\section{Two simple proofs}

\subsection{Proofs using reflection principles}
\label{sec:2}
In this section we give a simple proof of Theorem~A based on the
reflection principle for harmonic functions and its extension to the Helmholtz equation. At the same time we
remove the artificial condition in \cite{Le,Sj,JK,GJ} where the
solution of the Helmholtz equation  is assumed to be  the Fourier transform of a
measure. 


\begin{theorem} \label{refl1}
  Let $d\geq 2$, $k\in \R$,  and $\Omega$ be a domain (an open,
  connected set)  in $\R^d$ with $0\in\Omega$.
Let $\theta_1,\theta_2\in\S^{d-1}$ be such that $\dst\arccos\scal{\theta_1,\theta_2}\notin\pi\Q$.
 Let $u$ be a solution of the Laplace-Helmholtz equation on $\Omega$
$$
\Delta u+k^2 u=0
$$
satisfying one of the following boundary conditions
$$
\left\{\begin{matrix}u=0&\mbox{on }\theta_1^\perp\cap\Omega\\
u=0&\mbox{on }\theta_2^\perp\cap\Omega\end{matrix}\right.
\quad\mbox{or}\quad
\left\{\begin{matrix}u=0&\mbox{on }\theta_1^\perp\cap\Omega\\
\partial_nu=0&\mbox{on }\theta_2^\perp\cap\Omega\end{matrix}\right..
$$
Then $u=0$.
\end{theorem}

\begin{proof}
Let $r$ be such that the ball $\overline{B(0,r)}\subset\Omega$. As $u$ is real analytic, it is enough to show that
$u=0$ on $B(0,r)$.

%
For $j=1,2$, let $R_j$ be the reflection with respect to the hyperplane $\theta_j^\perp$.
The Schwarz reflection principle (see, e.g.,  \cite{DL}) implies that  $u(R_j x)=-
u(x)$ in the case of the Dirichlet condition $u \Big|_{H_j} = 0$ and $u(R_j x)=
u(x)$ in the case of the Neumann  condition $\partial _n u \Big|_{H_j}
= 0$.  It follows that $u((R_1R_2)^{2n} x)= u(x)$.



In particular, if $u=0$ on $\theta_1 ^\perp\cap B(0,r)$, then
$u\bigl((R_1R_2)^{2n}x)=0$ for every $x\in \theta_1 ^\perp \cap B(0,r)$. 
But $R_1R_2$ is a rotation by the  angle 
$2\arccos\scal{\theta_1,\theta_2}$ in the affine plane $\vect(\theta_1,\theta_2)$. As the angle is an irrational multiple of
$\pi$, the orbit of $\theta_1^\perp\cap B(0,r)$ under $ (R_1R_2)^2$ is dense in $B(0,r)$, thus $u=0$ on a dense set and, 
as $u$ is continuous, $u=0$ everywhere.
\end{proof}

Theorem~\ref{refl1} assumes either two Dirichlet conditions or a mixture of
Dirichlet and Neumann conditions. For the case of two Neumann conditions
it is not difficult to see that, if $u$ is radial, then
it satisfies two Neumann conditions,
$\left\{\begin{matrix}\partial_nu=0&\mbox{on }\theta_1 ^\perp \\
\partial_nu=0&\mbox{on }\theta_2 ^\perp \end{matrix}\right.$ on two
arbitrary hyperplanes $\theta _1 ^\perp$ and $\theta _2 ^\perp$.  

Moreover, in dimension 2, the above proof shows that only radial functions can occur. It is then not difficult to see that
$u$ is constant if it satisfies the Laplace equation $\Delta u=0$ and that $u$ is a constant multiple of
$J_0(k|x|)$ if $u$ is a solution of the Helmholtz equation $\Delta u+k^2u=0$.

\subsection{The non-crossing lemma and applications}
\label{Sec:3}

In the previous section we used reflection principles
to show that a  solution of the Laplace and Helmholtz equation can not
vanish on two lines that intersect with 
irrational angle. When replacing lines by more general curves,
reflection principles become substantially more involved~\cite{EK,Sa}. In
order to tackle that case, we develop 
an alternative proof that is based on Fourier expansions 
in the angular variable in polar coordinates.
The results will follow from  the following lemma.

\begin{lemma}[Non-crossing lemma]
Let $\Gamma_1,\Gamma_2$ be two $\cc^1$-curves in the plane intersecting at $0$.
Let $\gamma^\prime_1$ (resp.\ $\gamma^\prime_2$) be the vector tangent to $\Gamma_1$ (resp.\ $\Gamma_2$)
at $0$ and assume that the angle between $\gamma^\prime_1$ and $\gamma^\prime_2$ is not
a rational multiple of $\pi$, $\dst\arccos\scal{\gamma^\prime_1,\gamma^\prime_2}\notin\pi\Q$.

Let $k_{m}$ be  a strictly increasing sequence of non-negative
numbers and  $(\varphi_m)_{m\geq0}$ be a sequence of 
functions such that $\varphi_m(r)=r^{k_{m}}\bigl(1+o(1)\bigr)$ 
uniformly in a fixed  neighborhood of $0$.   Let $u$ be a function on $\R^2$ given by an expansion in polar coordinates
of the form
\[
u(r,\theta)=\sum_{m\in\mathbb{Z}} c_m\varphi_{|m|}(r)e^{im\theta},
\] 
with $\sum_{m\in\Z}|c_m|r_0^{k_{|m|}}<+\infty$ for some $r_0>0$.

If $u$ vanishes on $\Gamma_1$ and on $\Gamma_2$, then $u\equiv0$
in $\{z\in\C\,: |z|<r_0\}$.
\end{lemma}

\begin{remark}
The picture below illustrates the conditions in the lemma. Also, the curves do not need to intersect, it is enough that
they are rays emanating from the origin. The lemma is typically
applied to a function $u$ that is real analytic in a domain $\Omega$ 
and that has a representation of the above form in the neighborhood of
$0$. In that case we can even conclude that $u$ vanishes in all of $\Omega$.

\begin{figure}[!ht]
\begin{center}
\setlength{\unitlength}{0.5cm}
\begin{picture}(9,9)
\qbezier(0,4.5)(0,9)(4.5,9)
\qbezier(4.5,9)(9,9)(9,4.5)
\qbezier(9,4.5)(9,0)(4.5,0)
\qbezier(4.5,0)(0,0)(0,4.5)
\put(2,7){$\Omega$}
\put(3.8,4){$0$}
\qbezier(4.5,4.5)(7,8.5)(8,7)
\put(4.5,4.5){\vector(1,2){2}}
\put(5.3,8.1){$\gamma^\prime_2$}
\put(7.1,6.1){$\Gamma_2$}
\qbezier(4.5,4.5)(7,4.5)(7.5,2.5)
\put(4.5,4.5){\vector(1,0){4}}
\put(8,5){$\gamma^\prime_1$}
\put(7.6,2.2){$\Gamma_1$}
\end{picture}
\caption{In this lemma $\gamma^\prime_1=(1,0)$ and $\gamma^\prime_2=(1,2)$.}
\label{fig:cor3}
\end{center}
\end{figure}
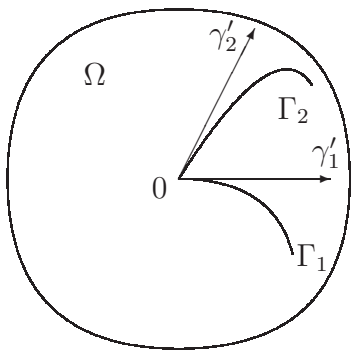
\end{remark}


\begin{proof}
Without loss of generality, we may parametrize the curves in polar
coordinates by the radius $r$ so that 
$\gamma _1(r) = (r,\theta_1(r))$ and $\gamma _2(r) = (r,\theta_2(r))$ with $\theta_1(0)=0,\
\theta_2(0)=\eta$ and $\eta = \arccos{\langle \gamma_1', \gamma
  _2'\rangle }$. (If necessary, we shrink the given neighborhood of $0$
  so that $\theta _j(t)$ is well-defined.) Since the   curves
  intersect at the origin, we must have $u(0,0)=c_0=0$.

  Now  assume that $c_{m}=0$ for $m=0,\pm1,\dots,\pm (n-1)$, then by
  induction we have 
\begin{eqnarray*}
u(r,\theta)&=&\sum_{|m|\geq n} c_m\varphi_{|m|}(r)e^{im\theta}=
\big(c_ne^{in\theta}+c_{-n}e^{-in\theta}\big)r^{k_n}\big(1+o(1)\big)+\sum_{|m|>n}c_m\varphi_m(r)e^{im\theta}\\
&=&\big(c_ne^{in\theta}+c_{-n}e^{-in\theta}\big)r^{k_n}+o(r^{k_n})
\end{eqnarray*}

In other words, we have 
$$
\lim _{r \to 0} \frac{u\big(r,\theta_j(r)\big)}{r^{k_n}} = 0$$
 for  $j=1,2$, and the hypothesis implies that
\begin{equation}
\label{eq:cross}
\left\{
\begin{matrix}
c_n&+&c_{-n}&=&0 , \\ c_ne^{in\eta}&+&c_{-n}e^{-in\eta}&=&0 .
\end{matrix}\right.
\end{equation}
This is  a system of two equations with two unknowns, and the
determinant of its coefficient matrix is $-2i\sin(n\eta)\not=0$ 
since $\eta\notin\pi\Q$. Hence, $c_{\pm n}=0$ and by induction we
conclude that  $c_n =0$ for all $n\in \Z $,  thus $u=0$. 
\end{proof}

\begin{remark}
\label{rem:cross}
 The condition $u=0$ on $\Gamma_2$ may be replaced by $\partial_\theta u=0$ on $\Gamma_2$.
Then  $\partial_\theta u(r,\theta)=\sum_{m\in\mathbb{Z}} imc_m\varphi_{|m|}(r)e^{im\theta}$,
and the system \eqref{eq:cross} is replaced by
$$
\left\{\begin{matrix}
c_n&+&c_{-n}&=&0\\ nc_ne^{in\eta}&-&nc_{-n}e^{-in\eta}&=&0  .
\end{matrix}\right.
$$
Since the  determinant is  $-2n\cos n\eta\not=0$ for $n\not=0$,  the
induction procedure is the same. 
\end{remark}

\begin{corollary}\label{cor:33}
Let $\Omega$ be a domain of $\R^2$.
Let $\Gamma_1,\Gamma_2$ be two $\cc^1$-curves in the plane intersecting at some $\omega\in\Omega$.
Let $\gamma^\prime_1$ (resp.\ $\gamma^\prime_2$) be the unit vector tangent to $\Gamma_1$ (resp. $\Gamma_2$)
at $\omega$ and assume that the angle between $\gamma^\prime_1$ and $\gamma^\prime_2$ is not
a rational multiple of $\pi$, $\dst\arccos\scal{\gamma^\prime_1,\gamma^\prime_2}\notin\pi\Q$.

Let $g$ be a continuous radial function
and define $h$ by $h(x)=g(x-\omega)$.
Then the only solution $u$ of $\Delta u+h\,u=0$ in $\Omega$ such that $u=0$ on $\Gamma_1\cup \Gamma_2$
is $u=0$ in $\Omega$.
\end{corollary}

\begin{proof} Without loss of generality, we may assume that $\omega=0$.
As $u$ is real-analytic, it is enough to prove that $u$ is $0$ in a neighborhood of $0$ in $\Omega$.
In such a neighborhood, $u$ has a Fourier expansion of the form
$$
u(r \cos\theta , r \sin \theta )=\sum_{m\in\Z} c_me^{im\theta}\ffi _{|m|}(r),
$$
where $\ffi_m$ is the unique normalized  solution of
$$
\ffi_m^{\prime\prime}(r)+\frac{1}{r}\ffi_m^{\prime}(r)
+\left(h(r)-\frac{m^2}{r^2}\right) \ffi _m(r) = 0 
$$
that is smooth at $0$. It is not hard to see that
$\ffi_m(r)\simeq r^{m}$ when $r\to0$.
Thus the non-crossing lemma applies and yields the stated  conclusion. 
\end{proof}

\section{Robin type conditions} \label{sec:3}

We next study the Helmholtz equation with Robin-type conditions on two
intersecting curves $\Gamma _1$ and $\Gamma _2$ in $\R ^2$. 
\subsection{The main result}
We write the solution of
the Helmholtz-Laplace  equation in polar coordinates
$$
u_p (r,\theta ) = u(r \cos \theta , r \sin \theta ) \, .
$$
Recall that the gradient
in polar coordinates is given by $\nabla F=(\partial_r F, r^{-1}\partial_\theta F)$.

We  parametrize the curves in polar coordinates as $\gamma (r) = (r
\cos \theta (r) , r \sin \theta (r))$ with the radius $r$  as the variable. By $\partial _t
u_p$ and $\partial _n u_p$ we denote the tangential and normal
derivatives of $u_p$ along a curve. A simple computation shows that
\begin{equation}
    \label{3.4a}
\partial_{t}u_p\bigl(r,\theta(r)\bigr)
=\partial_{r}u_p\bigl(r,\theta(r)\bigr)
+\theta^\prime(r)\partial_{\theta}u_p\bigl(r,\theta(r)\bigr)
  \end{equation}
and
\begin{equation}
  \label{eq:xy}
\partial_{n}u_p\bigl(r,\theta(r)\bigr)
=-r\theta^\prime(r)\partial_{r}u_p\bigl(r,\theta(r)\bigr)
+\frac{1}{r}\partial_{\theta}u_p\bigl(r,\theta(r)\bigr) \, .
\end{equation}
We will work with Robin type
conditions of the form $\alpha u_p + \beta \partial _t u_p + \tilde
\beta \partial _n u_p = 0$ on two  curves $\Gamma _1$ and $\Gamma _2$
and allow the coefficients $\alpha , \beta , \tilde \beta $ to be
functions of $r$.

\begin{theorem}\label{th:Robin2}
Let $\Omega\subset\mathbb{R}^2$ be a domain with $B(0,r_0)\subset\Omega$
for some $r_0>0$. Let $k\in\R$ and for $j=1,2$ let
$\alpha_j,\beta_j,\tilde\beta_j$ be $\cc^1$-functions
$(-r_0,r_0)\to\R$.

When $\beta_j(0)=\tilde\beta_j(0)=0$, we impose the following non-degeneracy:
$\alpha_j(0)\not=0$ and $\beta_j(r)=\tilde\beta_j(r)=o(r)$
when $r\to0$.

Next define $\ffi_j=0$ if $\beta_j(0)=\tilde\beta_j(0)=0$ and $\ffi_j=\arg\bigl(\beta_j(0)+i\tilde\beta_j(0)\bigr)$
otherwise.

Let $\theta_1,\theta_2$ be two $\cc^1$-functions $(-r_0,r_0)\to\R$
and assume that $\ffi_1-\ffi_2\notin
\bigl(\theta_1(0)-\theta_2(0)\bigr)\Z+\pi\Z$. 
Let $\Gamma_j=\{ (r \cos \theta_j(r), r\sin \theta _j(r)): \ r\in(-r_0,r_0)\}$ be the corresponding curves in polar coordinates in $\Omega$.

Let $u$ be a solution of the 
Helmholtz-Laplace equation
\begin{equation}
\label{eq:helmrob}
\Delta u(x,y)+k^2 u(x,y)=0, \quad \quad (x,y)\in\Omega,
\end{equation}
with Robin-type  conditions
\begin{equation}
	\label{eq:Robinthm}
\alpha_j(r) u\bigl(r,\theta_j(r)\bigr)
+\beta_j(r)\partial_tu\bigl(r,\theta_j(r)\bigr)
+\tilde\beta_j(r)\partial_nu\bigl(r,\theta_j(r)\bigr)=0
\end{equation}
on $\Gamma_j$ for $j=1,2$.
If $u(0,0)=0$ then $u=0$.
\end{theorem}

\begin{remark} 
The condition $\ffi_1-\ffi_2\notin
\bigl(\theta_1(0)-\theta_2(0)\bigr)\Z+\pi\Z$ can be seen as either
excluding a countable set 
of initial conditions on $\ffi_1-\ffi_2$ or a countable set of
directions $\theta_1(0)-\theta_2(0)$ at the origin. 
For instance, if $\beta_1(0)=\beta_2(0)=0$ or
$\tilde\beta_1(0)=\tilde\beta_2(0)=0$ we obtain the condition
$\theta_1(0)-\theta_2(0)\notin \Q\pi$. 

The statement would be more natural with
$\alpha_j,\beta_j,\tilde\beta_j$ constants. However, in order to
apply  Theorem~\ref{th:Robin2} to  the Helmholtz equation
on manifolds, we need this more general condition. 
Also, the non-degeneracy is then needed, since any solution of course
satisfies a Robin-type condition of the type \eqref{eq:Robinthm}.

\end{remark}

\begin{proof}
First note that, as $u$ is real-analytic,  it is enough to show that $u=0$ in $B(0,r_0)$.

\medskip

\noindent{\bf Case 1: $k=0$.} We first consider the case $k=0$ corresponding to the Laplace equation.
Locally 
in a neighborhood of $0$ and in polar coordinates, its solution has
the following series  expansion~\cite{MF} (with
$u_p(r,\theta)=u(r\cos\theta,r\sin\theta)$): 
\begin{equation}
	\label{eq:pollap}
u_p(r,\theta)=\sum_{m\in\mathbb{Z}}c_m r^{|m|}e^{im\theta}.
\end{equation}
Let us note right away that $c_0=u(0,0)=0$.

It follows that
\begin{eqnarray*}
	\partial_{r}u_p(r,\theta)&=&
	\sum_{m=1}^\infty m\bigl(
	c_m e^{im\theta}+c_{-m} e^{-im\theta}\bigr)r^{m-1}, \\
    \frac{1}{r}\partial_{\theta}u_p(r,\theta)&=&
   \sum_{m=1}^\infty im\bigl(
	c_m e^{im\theta}-c_{-m} e^{-im\theta}\bigr)r^{m-1} .
\end{eqnarray*}

Now suppose  that $c_{-m_0}=\cdots=c_{m_0}=0$ for some $m_0\geq0$. Then
by looking at the lowest order terms in $r$ we see from \eqref{3.4a}
and  \eqref{eq:xy}  that 
\begin{align} \label{hel1}
u_p\bigl(r,\theta_j(r)\bigr)
&=\bigl(c_{m_0+1}e^{i(m_0+1)\theta _j}+c_{-(m_0+1)}e^{-i(m_0+1)\theta
_j}\bigr)r^{m_0+1}+o(r^{m_0+1})
\\
\partial_{t}u_p\bigl(r,\theta_j(r)\bigr)
&= (m_0+1) \bigl(c_{m_0+1}e^{i(m_0+1)\theta
  _j}+c_{-(m_0+1)}e^{-i(m_0+1)\theta _j}\bigr)r^{m_0}+o(r^{m_0})
\notag \\
\partial_{n}u_p\bigl(r,\theta_j(r)\bigr)
&=i (m_0+1)
\bigl(c_{m_0+1}e^{i(m_0+1)\theta _j }-c_{-(m_0+1)}e^{-i(m_0+1)\theta
  _j }\bigr)r^{m_0}+o(r^{m_0}).   \notag 
\end{align}
It follows that
\begin{align*}
\alpha_j(r)u_p\bigl(r,\theta_j(r)\bigr)
&+\beta_j(r)\partial_{t}u_p\bigl(r,\theta_j(r)\bigr)
+\tilde\beta_j(r)\partial_{n}u_p\bigl(r,\theta_j(r)\bigr)\\
&=\alpha_j(r)\bigl(c_{m_0+1}e^{i(m_0+1)\theta _j
}+c_{-(m_0+1)}e^{-i(m_0+1)\theta _j}\bigr)r^{m_0+1}\\
  &+ (m_0+1)\bigl(\beta_j(r)+i\tilde
  \beta_j(r))c_{m_0+1}e^{i(m_0+1)\theta _j }r^{m_0}\\
 &+ (m_0+1)\bigl(\beta_j(r)-i\tilde
 \beta_j(r))c_{-(m_0+1)}e^{-i(m_0+1)\theta _j }r^{m_0} \\
 &+o(r^{m_0+1})+o\Bigl(\bigl(\beta_j(r)+i\tilde
 \beta_j(r))r^{m_0}\Bigr) .
\end{align*}

We now distinguish two cases: 


-- either $\bigl(\beta_j(0),\tilde \beta_j(0)\bigr)\not=(0,0)$ and then
$\kappa_j:=\beta_j(0)+i\tilde \beta_j(0)\not=0$ so that
\begin{multline*}
\alpha_j(r)u_p\bigl(r,\theta_j(r)\bigr)
+\beta_j(r)\partial_{t}u_p\bigl(r,\theta_j(r)\bigr)
+\tilde\beta_j(r)\partial_{n}u_p\bigl(r,\theta_j(r)\bigr)\\
= (m_0+1)\bigl(c_{m_0+1}\kappa_j
e^{i(m_0+1)\theta _j}+c_{-(m_0+1)}\overline{\kappa_j}
e^{-i(m_0+1)\theta _j }\bigr)r^{m_0}+o(r^{m_0}) , 
\end{multline*}

-- or $\bigl(\beta_j(0),\tilde \beta_j(0)\bigr)=(0,0)$, and then we assumed that
$\beta_j(r)\pm i\tilde \beta_j(r)=o(r)$,   thus
\begin{multline*}
\alpha_j(r)u_p\bigl(r,\theta_j(r)\bigr)
+\beta_j(r)\partial_{t}u_p\bigl(r,\theta_j(r)\bigr)
+\tilde\beta_j(r)\partial_{n}u_p\bigl(r,\theta_j(r)\bigr)\\
=\alpha_j(0)\bigl(c_{m_0+1}e^{i(m_0+1)\theta _j
}+c_{-(m_0+1)}e^{-i(m_0+1)\theta _j }\bigr)r^{m_0+1}
+o(r^{m_0+1}).
\end{multline*}

It follows that each  Robin condition
$$
\alpha_j(r)u_p\bigl(r,\theta_j(r)\bigr)
+\beta_j(r)\partial_{t}u_p\bigl(r,\theta_j(r)\bigr)
+\tilde\beta_j(r)\partial_{n}u_p\bigl(r,\theta_j(r)\bigr)
=0
$$
implies that 
$$
\begin{cases}
	c_{m_0+1}\kappa_j e^{i(m_0+1)\theta_j(0)}+c_{-(m_0+1)}\overline{\kappa_j} e^{-i(m_0+1)\theta_j(0)}=0
	&\mbox{if }   \kappa _j \neq 0 , \\ 
	c_{m_0+1}e^{i(m_0+1)\theta_j(0)}+c_{-(m_0+1)}e^{-i(m_0+1)\theta_j(0)}=0
	&\mbox{if } \kappa _j = 0 . 
\end{cases}
$$
Now, if we impose  two Robin conditions, then every pair 
$c_{-(m_0+1)},c_{m_0+1}$ is the  solution of  one of the following linear
systems.  

(i)  If $\bigl(\beta_j(0),\tilde \beta_j(0)\bigr)\not=(0,0)$ for $j=1$ and $j=2$, then
$$
\left\{
\begin{matrix}
	c_{m_0+1}\kappa_1
        e^{i(m_0+1)\theta_1(0)}&+&c_{-(m_0+1)}\overline{\kappa_1}
        e^{-i(m_0+1)\theta_1(0)}&=&0 ,\\
	c_{m_0+1}\kappa_2
        e^{i(m_0+1)\theta_2(0)}&+&c_{-(m_0+1)}\overline{\kappa_2}
        e^{-i(m_0+1)\theta_2(0)}&=&0 , 
\end{matrix}\right.
$$
which has determinant $2i \mathrm{Im}\ \kappa _1 \overline{\kappa _2}
e^{i (m_0+1) (\theta _1 - \theta _2)} =
  2i|\kappa_1\kappa_2|\sin\bigl(\varphi _1 - \varphi _2 +
  (m_0+1)\bigl(\theta_1(0)-\theta_2(0)\bigr)$.  This determinant is
  never $0$ by the assumption on $\theta_1(0)-\theta_2(0)$. Therefore
  $c_{m_0+1}=c_{-(m_0+1)}=0$. 

(ii)  If $\bigl(\beta_j(0),\tilde \beta_j(0)\bigr)=(0,0)$ for $j=1,2$,
then 
$$
\left\{
\begin{matrix}
	c_{m_0+1}
        e^{i(m_0+1)\theta_1(0)}&+&c_{-(m_0+1)}e^{-i(m_0+1)\theta_1(0)}&=&0
        , \\
	c_{m_0+1}
        e^{i(m_0+1)\theta_2(0)}&+&c_{-(m_0+1)}e^{-i(m_0+1)\theta_2(0)}&=&0 .
\end{matrix}\right.
$$
Again the determinant of this system is $2i\sin (m_0+1)\bigl(\theta_1(0)-\theta_2(0)\bigr)\not=0$
by our assumption on $\theta_1(0)-\theta_2(0)$. Therefore $c_{m_0+1}=c_{-(m_0+1)}=0$.

(iii)  If $\bigl(\beta_1(0),\tilde \beta_1(0)\bigr)=(0,0)$ and $\bigl(\beta_2(0),\tilde \beta_2(0)\bigr)\not=(0,0)$
(without loss of generality), then 
$$
\left\{
\begin{matrix}
	c_{m_0+1}
        e^{i(m_0+1)\theta_1(0)}&+&c_{-(m_0+1)}e^{-i(m_0+1)\theta_1(0)}&=&0
        , \\
	c_{m_0+1}\kappa_2
        e^{i(m_0+1)\theta_2(0)}&+&c_{-(m_0+1)}\overline{\kappa_2}
        e^{-i(m_0+1)\theta_2(0)}&=&0 , 
\end{matrix}\right.
$$
which has determinant $ 2i \mathrm{Im}\, \overline{\kappa _2}
  e^{i(m_0+1) (\theta _1(0) - \theta _2(0))} = 2i|\kappa_2|\sin\bigl( -\phi _2 + (m_0+1)\bigl(\theta_1(0)-\theta_2(0)\bigr)\not=0$
by assumption on $\theta_1(0)-\theta_2(0)$. Therefore $c_{m_0+1}=c_{-(m_0+1)}=0$.

In all three  cases, we obtain  that $c_{m_0+1} = c_{-(m_0+1)}=  0$, and by
induction,  it follows that $c_m=0$ for every $m$, thus $u=0$.

\medskip

\noindent{\bf Case 2: $k\neq 0$.}
We next treat the Helmholtz equation with Robin-type
conditions on two curves. After a change of variables,  we may assume
without loss of generality that $k=1$. Then 
in a neighborhood of $0$,   the solution written  in polar coordinates
possesses a Bessel expansion of the form 
\begin{equation}
	\label{eq:polhelm}
u_p(r,\theta)=\sum_{m\in\mathbb{Z}}c_m J_{|m|}(r)e^{im\theta},
\end{equation}
where $J_m$ is the Bessel function of the first kind and order $m$.

Since $J_m(0) = 0$ for $m\neq 0$, we obtain  that $c_0=u(0,0)=0$.

Using the relation for the Bessel functions
\[
\frac{m}{r}J_m(r)=\frac{1}{2}(J_{m+1}(r)+J_{m-1}(r)),
\]
we obtain
\begin{eqnarray*}
\frac{1}{r}\partial_\theta u_p(r,\theta)&=&
\sum_{m\ge 1}(c_m e^{im\theta}-c_{-m}e^{-im\theta})\frac{im}{r}J_m(r)\\
&=&\sum_{m\ge 1}(c_m e^{im\theta}-c_{-m}e^{-im\theta})\frac{i}{2}(J_{m+1}(r)+J_{m-1}(r))\\
&=&\frac{i}{2}\sum_{m\ge 0}(c_{m+1} e^{i(m+1)\theta}-c_{-m-1}e^{-i(m+1)\theta})J_m(r)\\
&&+\frac{i}{2}\sum_{m\ge 2}(c_{m-1} e^{i(m-1)\theta}-c_{-m+1}e^{-i(m-1)\theta})J_m(r)\\
&=&\frac{i}{2}\bigl(c_1e^{i\theta}-c_{-1}e^{-i\theta}\bigr)J_0(r)+\frac{i}{2}\bigl(c_2e^{2i\theta}-c_{-2}e^{-2i\theta}\bigr)J_1(r)\\
&&+\frac{i}{2}\sum_{k\ge 2}\Bigl(c_{k+1}e^{i(k+1)\theta}+c_{k-1}e^{i(k-1)\theta}\\
&&\qquad-c_{-(k-1)}e^{-i(k-1)\theta}-c_{-(k+1)}e^{-i(k+1)\theta}\Bigr)J_k(r).
\end{eqnarray*}

In the same way, using the relation for the Bessel functions
\[
\frac{\mathrm{d}}{\mathrm{d}r}J_m(r)=
\frac{1}{2}(J_{m+1}(r)-J_{m-1}(r))
\]
for  $m\not=0$, we obtain
\begin{eqnarray*}
\partial_r u_p(r,\theta)&=&
\frac{1}{2}
\sum_{m\ge 1}(c_m e^{im\theta}+c_{-m}e^{-im\theta})
\bigl(J_{m+1}(r)-J_{m-1}(r)\bigr)\\
&=&-\frac{1}{2}(c_1 e^{i\theta}+c_{-1}e^{-i\theta})J_0(r)
-\frac{1}{2}(c_2 e^{2i\theta}+c_{-2}e^{-2i\theta})
J_1(r)\\
&&+\frac{1}{2}
\sum_{k\ge 2}(-c_{k+1} e^{i(k+1)\theta}+c_{k-1} e^{i(k-1)\theta}\\
&&\qquad\qquad+c_{-(k-1)}e^{-i(k-1)\theta}-c_{-(k+1)} e^{-i(k+1)\theta})J_k(r).
\end{eqnarray*}

Recall that $J_m(r)=\dst\frac{r^m}{2^m m!}+o(r^m)$ when $r\to 0$.
Suppose now that $c_{-m_0}=\cdots=c_{m_0}=0$ for some $m_0\geq0$, then
the terms of lowest order in $r$ in \eqref{3.4a} and  \eqref{eq:xy} yield 
\begin{multline*}
u_p\bigl(r,\theta_j(r)\bigr)
=\frac{1}{2^{m_0+1}(m_0+1)!}\bigl(c_{m_0+1}e^{i(m_0+1)\theta
  _j}+c_{-(m_0+1)}e^{-i(m_0+1)\theta _j }\bigr)r^{m_0+1}+o(r^{m_0+1})
\\
\partial_{t}u_p\bigl(r,\theta_j(r)\bigr)
=
\frac{-1}{2^{m_0+1}m_0!}\bigl(c_{m_0+1}e^{i(m_0+1)\theta _j
}+c_{-(m_0+1)}e^{-i(m_0+1)\theta _j }\bigr)r^{m_0}+o(r^{m_0})
\\
\partial_{n}u_p\bigl(r,\theta_j(r)\bigr)
=\frac{i}{2^{m_0+1}m_0!}\bigl(c_{m_0+1}e^{i(m_0+1)\theta
  _j}-c_{-(m_0+1)}e^{-i(m_0+1)\theta _j }\bigr)r^{m_0}+o(r^{m_0})
\, .
\end{multline*}
Thus the terms of lowest order in $r$ are exactly a fixed multiple of
the terms in \eqref{hel1} for the Laplace equation. 

Consequently the Robin condition~\eqref{eq:Robinthm} implies the same systems of
equations as in the previous case. Thus the remainder of the proof is
exactly the same and we conclude that $u= 0$.  
\end{proof}

\begin{remark} \label{rem3}
If we replace the exact condition
	$$
	\alpha_j(r) u\bigl(r,\theta_j(r)\bigr)
+\beta_j(r)\partial_tu\bigl(r,\theta_j(r)\bigr)
+\tilde\beta_j(r)\partial_nu\bigl(r,\theta_j(r)\bigr)=0,
$$
by an approximate one
	$$
	\alpha_j(r) u\bigl(r,\theta_j(r)\bigr)
+\beta_j(r)\partial_tu\bigl(r,\theta_j(r)\bigr)
+\tilde\beta_j(r)\partial_nu\bigl(r,\theta_j(r)\bigr)=o(r^N),
$$
and assume that $(\beta_j(0), \tilde \beta _j(0)) \neq (0,0)$, then we may still conclude that
$c_{-N-1}=\cdots=c_{N+1}=0$ in the decompositions
\eqref{eq:pollap}-\eqref{eq:polhelm}.
As a consequence, $u(x)=o(\abs{x}^{N+1})$.
\end{remark}

The case of the Helmholtz equation on a manifold with Robin conditions
can be deduced 
by following Cheng's argument. 
\begin{corollary}
Let $\mathcal{M}$ be a two-dimensional $\cc^\infty$-Riemann manifold 
without boundary and $h\in\cc^\infty(M)$.

Let $\omega\in\mm$ and $\gamma_1,\gamma_2$ be two geodesics on $\mm$ that intersect at $\omega$ such that the angle of their tangents at $\omega$
is an irrational multiple of $\pi$.
Let $a_1,a_2,b_1,b_2\in\R$ such that
$(a_j,b_j)\not=(0,0)$. 

If $u$ is a solution of the Helmholtz equation $\Delta u+h u=0$ with $u(\omega)=0$
and satisfies the two Robin conditions
\begin{equation}
	\label{eq:Robincor}
a_j u+b_j\partial_n u=0\quad\mbox{on }\gamma_j,\quad j=1,2,
\end{equation}
then $u=0$.
\end{corollary}

\begin{proof}
	The proof follows the path of Cheng's result on nodal domains
	\cite[Thm.~2.2]{Ch}. First, we use normal coordinates on $\mm$
	which allows us to restrict  $u$ to  a small neighborhood of
        $0$, $u(0)=0$, such  that $\gamma_1,\gamma_2$ are two straight
        lines intersecting at $0$
with angle that is an irrational multiple of $\pi$.

Further, Bers' result~\cite{Be} implies that we can write 
$$
u(x,y)=p(x,y)+(x^2+y^2)^{m/2+\epsilon}q(x,y)
$$
where
$p$ is an harmonic polynomial of degree $m$ and $q$ is
smooth. 
 Further, $u$  cannot vanish to infinite order at $0$ unless $u=0$
uniformly. Thus we may assume that $p\not=0$ if $u\not=0$. 

One then easily sees  that the Robin condition \eqref{eq:Robincor}
implies a Robin condition of the type given by \eqref{eq:Robinthm}
with $\alpha_j(0)=a_j$, $\beta_j(0)=0$ and $\tilde\beta_j(0)=b_j$. Moreover,
if $b_j=0$ then $\beta_j(r)=\tilde\beta_j(r)=0$ so that the non-degeneracy condition is satisfied.

Now, according to Remark \ref{rem3}, this Robin condition implies that 
$u(x,y)=(x^2+y^2)^{m/2+\epsilon}q(x,y)$ and $p=0$. We thus conclude that $u=0$.
\end{proof}

\subsection{The condition $u(0,0)=0$}

One may ask whether the condition $u(0,0)=0$ is really needed in
Theorem \ref{th:Robin2}. This hypothesis was the basis for the
induction and therefore crucial to obtain $u=0$ in the previous
proof. 

We now  show that this condition may not always be removed, not even in
the simple case of constant coefficients in the Robin condition. 

\begin{proposition}\label{prop:Robin}
Let $\Omega\subset\mathbb{R}^2$ be a domain with $0\in\Omega$. Let $k\in\R$ and
$\alpha_1,\alpha_2,\beta_1,\beta_2\in\R$ with $\beta_1,\beta_2\not=0$.
Let $\theta_1,\theta_2\in\R$ be such that $\theta_1-\theta_2\notin \pi\mathbb{Q}$.
Let $l_{\theta_i}=\R(\cos\theta_j,\sin\theta_j)$ be lines passing through 0.
Then the space of solutions of the
Helmholtz equation with Robin boundary conditions:
\begin{equation}
\left\{\begin{array}{ll}\Delta u(x,y)+k^2 u(x,y)=0,& (x,y)\in\Omega,\\
\alpha_1 u(x,y)+\beta_1\partial_nu(x,y)=0,&(x,y)\in \ell_{\theta_1}\\
\alpha_2 u(x,y)+\beta_2\partial_nu(x,y)=0,&(x,y)\in \ell_{\theta_2},\\\end{array}\right.	
\label{systprop35}
\end{equation}
has dimension at most $1$.

The solution space has dimension exactly one, if  $\dst\frac{\theta_1-\theta_2}{\pi}$
is badly approximable by rationals in 
the sense that
there is a $c>\sqrt{5}$ such that for all but finitely many integers $k,\ell$,
\begin{equation}
	\label{eq:badapprox}
|k(\theta_1-\theta_2)-\ell\pi|\geq \frac{1}{ck} \, .
\end{equation}
\end{proposition}

\begin{remark}
The condition $c>\sqrt{5}$ comes from Hurwitz's theorem. If $\theta_1-\theta_2$ is an algebraic number of degree $2$, then
it is badly approximable by rationals.

As already noticed, if $\Delta u+k^2 u=0$ and $\partial_nu=0$ on $\ell_{\theta_1}$ and on $\ell_{\theta_2}$, then
$u$ is radial. It is well known that the only radial solutions of $\Delta u+k^2 u=0$ are constant when $k=0$
and constant multiples of $J_0(k\sqrt{x^2+y^2})$ when $k\not=0$. The above result extends this to more general Robin conditions.
\end{remark}

\begin{proof} 
We only treat the case $k=1$ and  leave the scaling to obtain $k\neq
0$ and  the easier case $k=0$ to
the reader. Again,    writing $u$
in polar coordinates as $u_p(r,\theta)=u(r\cos\theta,r\sin\theta)$,
the solution to the Helmholtz equation is of the form 
\[
u_p(r,\theta)=\sum_{m\in\mathbb{Z}}c_m J_{|m|}(r)e^{im\theta}.
\]
The previous proof  shows that the coefficients of $J_0$ and $J_1$ in
the Robin boundary condition for $j=1,2$ are given by 
\begin{eqnarray}
\alpha_jc_0+\frac{i\beta_j}{2}(c_1e^{i\theta_j}-c_{-1}e^{-i\theta_j})&=&0\label{eq:Robin1},\\
\alpha_j(c_1e^{i\theta_j}+c_{-1}e^{-i\theta_j})+\frac{i\beta_j}{2}(c_2e^{2i\theta_j}-c_{-2}e^{-2i\theta_j})&=&0,
\label{eq:Robin2}
\end{eqnarray}
and, for $m\geq 2$,
\begin{eqnarray}
\alpha_j( c_me^{im\theta_j}+c_{-m} e^{-im\theta_j})&&\nonumber\\
+\frac{i\beta_j}{2}\bigg(c_{m+1}e^{i(m+1)\theta_j}- c_{-m-1}e^{-i(m+1)\theta_j}&&\nonumber\\
\qquad+c_{m-1}e^{i(m-1)\theta_j}-c_{-m+1}e^{-i(m-1)\theta_j}\bigg)&=&0.\label{eq:Robin3}
\end{eqnarray}

Writing $\gamma_j=\dst 2i\frac{\alpha_j}{\beta_j}$,
\eqref{eq:Robin1} implies that
$$
\left\{\begin{matrix}
c_1e^{i\theta_1}&-&c_{-1}e^{-i\theta_1}&=&\gamma_1 c_0\\
c_1e^{i\theta_2}&-&c_{-1}e^{-i\theta_2}&=&\gamma_2 c_0 \, ,
\end{matrix}\right.
$$
from which follows 
$$
c_1=-\frac{\gamma_2e^{-i\theta_1}-\gamma_1e^{-i\theta_2}}{2i\sin (\theta_1-\theta_2)}c_0
\quad,\quad
c_{-1}=-\frac{\gamma_2e^{i\theta_1}-\gamma_1e^{i\theta_2}}{2i\sin (\theta_1-\theta_2)}c_0.
$$
Next, \eqref{eq:Robin2} implies
$$
\left\{\begin{matrix}
c_2e^{2i\theta_1}&-&c_{-2}e^{-2i\theta_1}&=&\gamma_1(c_1e^{i\theta_1}+c_{-1}e^{-i\theta_1})\\
c_2e^{2i\theta_2}&-&c_{-2}e^{-2i\theta_2}&=&\gamma_2(c_1e^{i\theta_2}+c_{-1}e^{-i\theta_2})
\end{matrix}\right.,
$$
that is
$$
\left\{\begin{matrix}
c_2e^{2i\theta_1}&-&c_{-2}e^{-2i\theta_1}&=&\frac{\gamma_1}{i\sin (\theta_1-\theta_2)}
\bigl(\gamma_1\cos(\theta_1-\theta_2)-\gamma_2\bigr)c_0 , \\
c_2e^{2i\theta_2}&-&c_{-2}e^{-2i\theta_2}&=&
\frac{\gamma_2}{i\sin
  (\theta_1-\theta_2)}\bigl(\gamma_1-\gamma_2\cos(\theta_1-\theta_2)\bigr)c_0 .
\end{matrix}\right.
$$

This may be written in the form
$$
\left\{\begin{matrix}
c_2e^{2i\theta_1}&-&c_{-2}e^{-2i\theta_1}&=&\frac{\mu_2}{\sin (\theta_1-\theta_2)}c_0\\
c_2e^{2i\theta_2}&-&c_{-2}e^{-2i\theta_2}&=&
\frac{\nu_2}{\sin (\theta_1-\theta_2)}c_0
\end{matrix}\right.,
$$
with $|\mu_2|,|\nu_2|\leq (|\gamma_1|+|\gamma_2|)^2$.
It follows that
$$
c_2=\frac{\mu_2e^{-2i\theta_2}-\nu_2e^{-2i\theta_1}}{2i\sin(\theta_1-\theta_2)\sin 2(\theta_1-\theta_2)}c_0
\quad,\quad
c_{-2}=-\frac{\nu_2e^{2i\theta_1}-\mu_2e^{2i\theta_2}}{2i\sin (\theta_1-\theta_2)\sin 2(\theta_1-\theta_2)}c_0.
$$

One can then write \eqref{eq:Robin3} in the form
\begin{eqnarray*}
c_{m+1}e^{i(m+1)\theta_j}- c_{-m-1}e^{-i(m+1)\theta_j}&=&
\gamma_j( c_me^{im\theta_j}+c_{-m} e^{-im\theta_j})\\
&&-c_{m-1}e^{i(m-1)\theta_j}+c_{-m+1}e^{-i(m-1)\theta_j},
\end{eqnarray*}
from which we get that $c_{\pm(m+1)}\sin (m+1)(\theta_1-\theta_2)$
is a linear combination of $c_{\pm m}$ and $c_{\pm(m-1)}$ with coefficients bounded by $\max(1,\gamma_1,\gamma_2)$
so that
$$
c_{\pm(m+1)}=\frac{\kappa_{\pm(m+1)}}{\prod_{k=1}^{m+1}\sin k(\theta_1-\theta_2)}c_0,
$$
with $|\kappa_{\pm(m+1)}|\leq (2+|\gamma_1|+|\gamma_2|)^{m+1}$.

Further, the Bessel's function satisfies the bound
\begin{equation}
\label{eq:bessel}
|J_{m}(r)|\leq \frac{r^{m}}{2^{m}m!}.
\end{equation}
Thus, if $\frac{\theta_1-\theta_2}{\pi}$ is badly approximable by rationals,
$$
\left(\prod_{k=1}^{m+1}\sin k(\theta_1-\theta_2)\right)^{-1}=O\bigl((m+1)!\bigr)
$$
so that
\begin{equation}
\label{eq:besselexp}
u(r,\theta)=\sum_m c_m J_{|m|}(r)e^{im\theta}
\end{equation}
converges to a solution $u$ of the Helmholtz equation.

\end{proof}

We suspect that in some cases the series \eqref{eq:besselexp} diverges so that the solution space is still of dimension 0.
For instance, if $\alpha_1=0$, $\beta_1=1$, $\theta_1=0$, that is if $u$ satisfies a Neumann equation on the $x$-axis,
then the Schwarz reflection principle shows that $c_{-m}=c_m$. Then, the Robin condition on $\ell_{\theta_2}$ implies
$$
\gamma_2c_m\cos m\theta_2=c_{m-1}\sin(m-1)\theta_2+c_{m+1}\sin(m+1)\theta_2.
$$
We think that for some numbers $\theta_2$ (in particular when $\theta_2/\pi$ is a Liouville number), this would lead to
a series \eqref{eq:besselexp} that is divergent.

\subsection{Stability}

A natural extension of Remark \ref{rem3} would be to replace
the Robin condition by an approximate Robin condition and ask whether
the corresponding solution is small. In general this is not true. 
We  illustrate the instability in the simplest case, namely for the Laplace equation with
Dirichlet conditions on lines. 

\begin{proposition} \label{prop:37}
Let $\theta_1,\theta_2\in\R$, be such that $\theta_1-\theta_2\notin
\pi\mathbb{Q}$, and let $\eps>0$.

 There exist infinitely many $n\in\N$
such that the solution of the Laplace  equation
\[
\left\{\begin{array}{ll}\Delta u(x,y)=0,& (x,y)\in D(0,1),\\
u(r,\theta_1)=\eps|r|^n,&|r|<1,\\
u(r,\theta_2)=2\eps|r|^n,&|r|<1,\end{array}\right.
\]
satisfies  $\norm{u}_{L^\infty (D(0,1))}\geq\frac{n\eps}{8}$.
\end{proposition}

\begin{proof}
Write $u$ in polar coordinates as 
$$
u(r,\theta)=\sum_{m\in\Z}c_mr^{|m|}e^{im\theta} = c_0 + \sum
_{m=1}^\infty (c_m e^{im\theta } + c_{-m} e^{-im\theta } ) r^m .
$$
Then  each Dirichlet condition is equivalent to the set of equations
$$
c_me^{im\theta_j}+c_{-m}e^{-im\theta_j}=u_{m,j},
$$
where $u_{n,1}=\eps$, $u_{n,2}=2\eps$ and $u_{m,j}=0$ if $m\not=n$.

For any $m $ the resulting  system of equations has  determinant $2i\sin
m(\theta_1-\theta_2)\not=0$. If   $m\not=n$, then  $c_m=c_{-m}=0$. 
For $m=n$ we obtain
$$
c_n=\frac{1-2e^{-in(\theta_1-\theta_2)}}{2i\sin n(\theta_1-\theta_2)}e^{-in\theta_2}\eps
\quad,\quad
c_{-n}=-\frac{1-2e^{in(\theta_1-\theta_2)}}{2i\sin n(\theta_1-\theta_2)}e^{in\theta_2}\eps.
$$
It follows that
$$
u(r,\theta)=\frac{\eps }{\sin n(\theta_1-\theta_2)}
\Im\big((1-2e^{-in(\theta_1-\theta_2)})e^{in(\theta-\theta_2)}\big) \,
r^n.
$$

But now, according to Dirichlet's theorem, there exist infinitely many
$n,p\in\Z\setminus\{0\}$ such that
$$
\abs{\frac{\theta_1-\theta_2}{\pi}n-p}\leq
\frac{1}{n} , 
$$
and thus $|\cos  n(\theta _1 - \theta _2) - \cos p\pi | \leq \pi
/n$ and $|\sin  n(\theta _1 - \theta _2) - \sin p\pi | \leq \pi
/n$.   

Now choose $\theta $ such that $n(\theta - \theta _2) = \pi / 2$ and
$e^{in(\theta - \theta _2)} = i$. Then for $n>6$
\begin{align*}
|\Im\big((1-2e^{-in(\theta_1-\theta_2)})e^{in(\theta-\theta_2)}\big)
&= |\mathrm{Re}\, ( 1-2e^{-in(\theta_1-\theta_2)})  | \\
& = |1 - 2 \cos
n(\theta _1 - \theta _2)| \geq |1 - 2(-1)^p| -
\frac{\pi}{n} \geq \frac{1}{2}\, .
\end{align*}


It follows that
$|u(r,\theta)|\geq \frac{n}{2\pi} \eps r^n$, from which the result follows immediately.
\end{proof}

Proposition~\ref{prop:37} implies that the reconstruction of a
harmonic function from its restriction to two lines is always
unstable. In particular, if $u_1$ and $u_2$ are harmonic functions
such that $\| (u_1-u_2)|_{\Gamma _j} \|_\infty < \epsilon $, we cannot
make any assertion on the error $\|u_1-u_2\|_{L^\infty (D(0,1))}$. 

Note that Dirichlet's theorem allows to simultaneously approximate
arbitrarily many irrational numbers. Consider the Laplace equation 
$$
\left\{\begin{matrix}\Delta u(x,y)=0,& (x,y)\in D(0,1),\\
u(r,\theta_j)=\alpha_j\eps|r|^n,&r\in\R,j=1,\ldots,N.\\
\end{matrix}\right. 
$$
For this system to have a solution,  the $\alpha_j$'s must satisfy some 
compatibility condition. Consequently, this system either has no
solution, or,  if the $\alpha _j$'s are compatible,  the norm of the
solution  can  be arbitrarily 
large.

\section{Results in higher dimension }

In this section we study possible extensions of the uniqueness
property to higher dimensions. The general problem is the
following: For $1\leq m<d$, 
let $\Gamma_1,\ldots,\Gamma_N$ be  $m$-dimensional analytic
submanifolds (or just linear subspaces) in $\R ^d$ 
intersecting at $0$.  Under which conditions is
$\Gamma _1 \cup \dots \cup \Gamma _N$ a set of uniqueness for
solutions of $\Delta  u + k^2 u = 0$? In other words, when is it true
that $u|_{\Gamma _1 \cup \dots \cup \Gamma _N} =0$ implies that $u =
0$. 

We will treat the case of two hypersurfaces ($m=d-1$) and give a
counter-example for an arbitrary finite collection of lines. 

\subsection{An extension of the results of Section \ref{Sec:3}}

We will now extend some of the results of Section \ref{sec:3} to higher
dimensions.

\begin{theorem}
Let $d\geq 3$ and let $\Omega$ be a domain in $\R^d$.
Let $\mm_1,\mm_2$ be two $d-1$ dimensional submanifolds of $\R^d$ that intersect at $0$ and let
$\theta_1,\theta_2$ be the unit normal vectors  to $\mm _1 $ and $\mm
_2$  at $0$. Assume that $\arccos\scal{\theta_1,\theta_2}\notin\pi\Q$.

If  $u$ is  a solution of $\Delta u+k^2u=0$ on $\Omega$ such that
$u=0$ on $\mm_1\cup\mm_2$, then $u=0$
on $\Omega$.
\end{theorem}

\begin{proof}
We first describe a particular basis of spherical harmonics 
and fix some notation taken from {\it e.g.} \cite{DX}.


First, we introduce the Gegenbauer polynomials
$$
C_n^\lambda(x)=\sum_{k=0}^{\lfloor n/2\rfloor}
(-1)^k\frac{\Gamma(n-k+\lambda )}{\Gamma(\lambda)k!(n-2k)!}(2x)^{n-2k}.
$$
They satisfy the orthogonality relation $\dst\int_0^\pi C_m^\lambda(\cos\theta)C_n^\lambda(\cos\theta)\sin(\theta)^{2\lambda}\d\theta=0$
if $m\not=n$.

We use spherical coordinates in $\R^d$,
$$
\left\{
\begin{array}{rcl}
x_1&=&r\sin\theta_{d-2}\cdots\sin\theta_1\cos\ffi\\
x_2&=&r\sin\theta_{d-2}\cdots\sin\theta_1\sin\ffi\\
x_3&=&r\sin\theta_{d-2}\cdots\sin\theta_2\cos\theta_1\\
\vdots&&\vdots\\
x_{d-1}&=&r\sin\theta_{d-2}\cos\theta_{d-3} \\
x_d&=&r\cos\theta_{d-2}
\end{array}\right.
$$
with $r\geq 0$, $\theta:=(\theta_1,\ldots,\theta_{d-2})\in [0,\pi)^{d-2}$, $\ffi\in[0,2\pi)$.

After a suitable translation and rotation, we may assume that, in a
neighborhood $V$ of $0$, both  $\mm_1\cap V$ and $\mm_2\cap V$
are parametrized as
$$
\mm_j\cap V=\bigl\{\bigl(r,\theta,\psi_j(r,\theta)\bigr):  0<r<\eps,\theta\in [0,\pi)^{d-2}\bigr\}
$$
where the $\psi_j$'s are smooth functions. Moreover, if we define
$$
\ffi_j=\psi_j(0,\theta)= \lim _{r\to 0} \psi _j (r,\theta )
$$
(which does not depend on $r$),
then our hypothesis on $\mm_1,\mm_2$ implies that $\ffi_1-\ffi_2\notin\pi\Q$.

Next, we  define a basis of spherical harmonics.
For $\alpha\in\nn:=\N_0^{d-2}\times\Z$ we set 
\begin{eqnarray*}
Y_\alpha(r,\theta_1,\ldots,\theta_{d-2},\ffi)&=&r^{|\alpha|}e^{i\alpha_{d-1}\ffi}
\prod_{j=1}^{d-2}(\sin\theta_{d-j})^{|\alpha|^{j+1}}C^{\lambda_j}_{\alpha_j}(\cos\theta_j)\\
&:= &r^{|\alpha|}e^{i\alpha_{d-1}\ffi}\tilde Y_{\alpha}(\theta_1,\ldots,\theta_{d-2})
\end{eqnarray*}
where $|\alpha|^j=\alpha_j+\cdots+\alpha_{d-1}$ and $\lambda_j=|\alpha|^{j+1}+(d-j-1)/2$.
Then $\{Y_\alpha :  \alpha\in\nn\}$ forms a basis of spherical harmonics on $\R^d$.
Further, for each $n\geq1$, the set
$$
\{\tilde Y_{\beta,m}\,: (\beta,m)\in\N_0^{d-1}, |\beta|+m=n\}
$$
is linearly independent.

In a neighborhood of $0$,
an arbitrary  solution $u$ of the Laplace equation $\Delta u=0$ can
then be written in spherical coordinates in the form 
\begin{multline}
\label{eq:lapdimd}
u(r,\theta_1,\ldots,\theta_{d-2},\ffi)=\sum_{\alpha\in\nn}c_{\alpha}Y_\alpha(r,\theta_1,\ldots,\theta_{d-2},\ffi)\\
=\sum_{n=0}^\infty r^n \sum_{m=-n}^n\bigg(\sum_{\substack{ \beta\in\N_0^{d-2}\\ |\beta|+|m|=n}}
c_{\beta,m}\tilde Y_{\beta,m}(\theta_1,\ldots,\theta_{d-2})\bigg)e^{im\ffi}
\end{multline}
The local solution of the Helmholtz equation $\Delta u+u=0$ can be written in the form~\cite{MF}
\begin{multline}
\label{eq:helmdimd}
u(r,\theta_1,\ldots,\theta_{d-2},\ffi)\\
=r^{-\frac{d-2}{2}}\sum_{n=0}^\infty J_{n+\frac{d-2}{2}}(r) \sum_{m=-n}^n\bigg(\sum_{\substack{ \beta\in\N_0^{d-2}\\ |\beta|+|m|=n}}
c_{\beta,m}\tilde Y_{\beta,m}(\theta_1,\ldots,\theta_{d-2})\bigg)e^{im\ffi}. 
\end{multline}
Now let $u$ be a solution of the Laplace equation, thus given by  \eqref{eq:lapdimd}.

First $c_{0}=u\bigl(0,\theta,\ffi_j\bigr)=0$. But then
$$
u\bigl(r,\theta,\psi _j(r,\theta)\bigr)=r
\sum_{m=-1}^1\bigg(\sum_{\substack{\beta\in\N_0^{d-2}\\ |\beta|+|m|=1}}
c_{\beta,m}\tilde Y_{\beta,m}(\theta)\bigg)e^{im\psi _j(r,\theta)}+o(r).
$$
As $e^{im\psi _j(r,\theta)}=e^{im\ffi_j}+o(1)$, we get
$$
\sum_{m=-1}^1\bigg(\sum_{\begin{matrix}\beta\in\N_0^{d-2}\\ |\beta|+|m|=1\end{matrix}}
c_{\beta,m}\tilde Y_{\beta,m}(\theta)\bigg)e^{im\ffi_j}=0.
$$
Thus for all $\theta \in [0,\pi )^{d-2}$
$$
(c_{0,1}e^{im\ffi_j}+c_{0,-1}e^{im\ffi_j})\tilde Y_{0,1}(\theta)
+\sum_{\begin{matrix}\beta\in\N_0^{d-2}\\ |\beta|=1\end{matrix}}
c_{\beta,0}\tilde Y_{\beta,0}(\theta)=0.
$$
The linear independence of the $\tilde Y_{\beta,m}$'s implies  that $c_{\beta,0}=0$ when $|\beta|=1$ and
$$
\left\{\begin{matrix}
c_{0,-1}e^{-i\ffi_1}&+&c_{0,1}e^{i\ffi_1}&=&0 , \\
c_{0,-1}e^{-i\ffi_2}&+&c_{0,1}e^{i\ffi_2}&=&0 .
\end{matrix}\right.
$$
As the determinant of this system is $2i\sin (\ffi_2-\ffi_1)\not=0$ we get $c_{0,-1}=c_{0,1}=0$.

Assume now that $c_{\beta,m}=0$ for every $\beta,m$ with $|\beta|+|m|\leq n_0-1$ then, as previously
$$
u\bigl(r,\theta,\psi _j(r,\theta)\bigr)=r^{n_0}
\sum_{m=-n_0}^{n_0}\bigg(\sum_{\begin{matrix}\beta\in\N_0^{d-2}\\ |\beta|+|m|=n_0\end{matrix}}
c_{\beta,m}\tilde Y_{\beta,m}(\theta)\bigg)e^{im\ffi_j}+o(r^{n_0}).
$$
It follows again that
\begin{multline}
\bigg(\sum_{\begin{matrix}\beta\in\N_0^{d-2}\\ |\beta|=n_0\end{matrix}}
c_{\beta,0}\tilde Y_{\beta,0}(\theta)\bigg)\\
+\sum_{m=1}^{n_0}\bigg(\sum_{\begin{matrix}\beta\in\N_0^{d-2}\\ |\beta|+m=n_0\end{matrix}}
(c_{\beta,m}e^{im\ffi_j}+c_{\beta,-m}e^{-im\ffi_j})\tilde Y_{\beta,m}(\theta)\bigg)=0.
\end{multline}
This implies that $c_{\beta,0}=0$ if $|\beta|=n_0$ and that for each $m=1,\ldots,n_0$
and each $\beta$ with $|\beta|=n_0-m$,
$$
\left\{\begin{matrix}
c_{\beta,-m}e^{-im\ffi_1}&+&c_{\beta,m}e^{im\ffi_1}&=&0 , \\
c_{\beta,-m}e^{-im\ffi_2}&+&c_{\beta,m}e^{im\ffi_2}&=&0 .
\end{matrix}\right.
$$
The determinant of this system is $2i\sin m(\ffi_2-\ffi_1)\not=0$,  thus $c_{\beta,-m}=c_{\beta,m}=0$.
We have thus proved the result.

The case of $u$ of the form \eqref{eq:helmdimd} is similar.

\end{proof}

\subsection{Negative results}
Finally we show that there is no unique continuation of the Helmholtz
equation for a finite set of intersecting lines. 

\begin{proposition}
\label{prop:line1}
Let $d\geq 3$ and $k\in\R$. For every
$\theta_1,\ldots,\theta_N\in\S^{d-1}$ there exists a non-zero solution of $\Delta u+k^2u=0$
such that $u=0$ on $\R\theta_1\cup\cdots\cup\R\theta_N$. 
\end{proposition}

\begin{proof}
It suffices  to find a non-zero spherical harmonic $Y$ of suitable
degree $m$ such that $Y(\pm \theta_j)=0$, $j=1,\ldots,N$. Then the
functions $u(r\theta)=r^mY(\theta)$ (for  $k=0$) and
$u(r\theta)=r^{-(d-2)/2}J_{m+(d-2)/2}(kr)Y(\theta)$ (for  $k\not=0$) vanish on the
lines $\R \theta _j, j=1, \dots , N$, and thus the solution of the
Helmholtz-Laplace  equation is not uniquely determined by its
restriction to a finite number of lines. 

Let $\mathcal{H}_m^d$ denote the subspace of spherical harmonics of
degree $m$ in $\R ^d$. Its dimension is $ \binom{m+d-1}{m}  \geq 2m+1$ for $d\geq 3$. 

Now consider the linear forms $L_j:\mathcal{H}_m^d\to\C$ given by
$L_j(Y)= Y(\theta_j)$ for $j=1, \dots , N$ and $L_j(Y) = Y(-\theta
_j)$ for $j=N+1, \dots , 2N$. A dimension count yields $\bigcap _{j}
\mathrm{ker}\, L_j \geq \mathrm{dim} \, \mathcal{H}^d_m - 2N \geq 1$
for $m$ large enough. So $ \bigcap _{j}
\mathrm{ker}\, L_j  \neq \{0\}$ and there exists a $Y\in
\mathcal{H}_m^d$ satisfying $Y(\theta _j) = 0$ for $j= 1 , \dots ,
N$. 



Thus the solution of the Helmholtz equation is not uniquely
determined by its restriction to any finite set of intersecting
lines. 
\end{proof}

We do not know what happens for restrictions  of the Helmholtz
equation  to 
$k$-dimensional subspaces when $2\leq k\leq d-2$ and 
$d\geq 4$.

Proposition~\ref{prop:line1} can be turned into a statement about
Heisenberg uniqueness pairs.

\begin{corollary}
The sphere $\S ^{d-1}$ for $d\geq 3$ and an arbitrary finite set $\R
\theta _j, j=1, \dots ,N$ of
lines through the origin cannot be a Heisenberg uniqueness pair.
There always  exist two distinct finite positive measures such that
$\mu_+,\mu_-$ are  supported on $\S^{d-1}$
 and $\widehat{\mu_+}|_{\R \theta _j}  = \widehat{\mu_-}|_{\R \theta _j} $ for $j=1,\ldots,N$.
\end{corollary}
\begin{proof}
Let $Y \in \mathcal{H} ^d_m$ be a non-zero spherical harmonic such that $Y(\pm
\theta _j) = 0 $ for $j=1, \dots ,N$ and define two measures $\mu
_{\pm}$ by 
$$
\mu_\pm=\left(1\pm\frac{Y(\theta)}{\norm{Y(\theta)}_{L^\infty(\S^{d-1})}}\right)\d\sigma(\theta).
$$
Clearly $\mu_{\pm}$ are positive and absolutely continuous with respect to the
surface measure on $\S ^{d-1}$. Furthermore, the Hecke-Funck Formula
\cite{SW} shows that $\widehat{\mu_+-\mu_-}(r\theta)=\dst
cr^{-(d-2)/2}J_{m+(d-2)/2}(r)Y(\theta)$ where $c$ 
is a non-zero  constant. Thus $\widehat{\mu_+}|_{\R \theta _j}  =
\widehat{\mu_-}|_{\R \theta _j} $. 

\end{proof}

\section*{Acknowledgments}
A.\,F.\,B. kindly acknowledge financial support from the IdEx postdoctoral program via the PDEUC project and from ERCEA Advanced Grant 2014 669689 - HADE.

Ph\,.J. kindly acknowledges financial support from 
the French-Tunisian CMCU/U\-ti\-que project 32701UB Popart.

The three authors kindly acknowledge the support of the Austrian-French Ama\-deus project 35598VB-ChargeDisq.

This study has been carried out with financial support from the French State, managed
by the French National Research Agency (ANR) in the frame of the ”Investments for
the future” Programme IdEx Bordeaux - CPU (ANR-10-IDEX-03-02).

\end{document}